\documentclass[11pt]{article}

\usepackage[english]{babel} %francais, polish, spanish, ...
\usepackage[utf8]{inputenc} % usually not needed (loaded by default)
\usepackage[T1]{fontenc}
\usepackage{amsmath,amssymb,fullpage}
\usepackage[all,dvips]{xy}

\usepackage{tikz}
\usepackage{tikz,fullpage}
\usetikzlibrary{positioning, arrows.meta}
\usetikzlibrary{matrix}

\usepackage{enumerate}
\usepackage{mathtools}

\usepackage{graphicx}
\usepackage{color}
\definecolor{vert}{rgb}{0.02,0.4,0.10}
\usepackage{hyperref}
\hypersetup{
	colorlinks=true,                % colorise les liens
	breaklinks=true,                % permet les retours Ã  la ligne pour les liens trop longs
	linkcolor= vert,                % couleur des liens internes aux documents (index, figures, tableaux, euations,...)
	citecolor= blue                % couleur des liens vers les references bibliographiques
}

\usepackage{graphicx}
\graphicspath{ {images/} }
\usepackage{pgf,tikz,pgfplots}
\pgfplotsset{compat=1.14}
\usepackage{mathrsfs}
\usetikzlibrary{arrows}

\usepackage{makeidx}
\usepackage{pstricks}
\providecommand{\otherindexspace}[1]{}

\usepackage{authblk}
\usepackage{amsthm}

\newtheorem{theorem}{Theorem}[section]

\newtheorem{proposition}[theorem]{Proposition}

\numberwithin{equation}{section}

\def\vp{\varepsilon}

\def \R{\mathbb {R}}

\def \Z{\mathbb{Z}}

\def\E{\mathbb{E}}
\def\P{\mathbb{P}}
\def\lb{[\![}
\def\rb{]\!]}

\usepackage{fancyhdr}
\usepackage{graphics}
\usepackage{graphicx}

\DeclareGraphicsExtensions{.eps,.bmp,.jpg,.pdf,.mps,.png,.gif}

\makeatletter
\def\titre{\@title}
\makeatother

\title{Gaskets of $O(2)$ loop-decorated random planar maps
}
\author{Emmanuel Kammerer \thanks{CMAP, \'Ecole polytechnique, Institut Polytechnique de Paris, 91120 Palaiseau, France, \textsf{emmanuel.kammerer@polytechnique.edu}
}}
\date{\today}
\begin{document}

\maketitle

\begin{abstract}
We prove that for $n=2$ the gaskets of critical rigid $O(n)$ loop-decorated random planar maps are $3/2$-stable maps. The case $n=2$ thus corresponds to the critical case in random planar maps. The proof relies on the Wiener--Hopf factorisation for random walks. Our techniques also provide a characterisation of weight sequences of critical $O(2)$ loop-decorated maps.
\end{abstract}

\begin{figure}[h]
	\centering
	\includegraphics[width=0.41\textwidth]{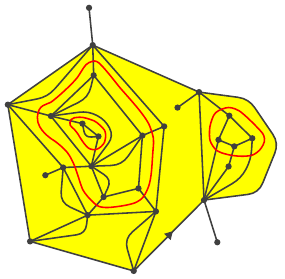}
	\hspace{0.8cm}
	\includegraphics[width=0.41\textwidth]{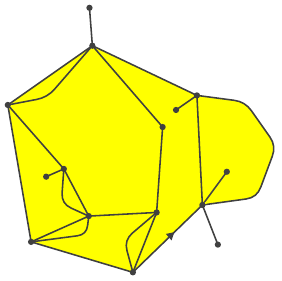}
	\caption{Left: a rigid loop-decorated bipartite planar map. Right: its gasket.}
	\label{carteO2}
\end{figure}

%\tableofcontents

\section{Introduction}
\subsection{$O(n)$ loop-decorated planar maps and their gaskets}
Random $O(n)$ loop-decorated random planar maps are a classical model which couples a statistical mechanics model with quantum gravity. Although they have been studied for more than thirty years by physicists and then by mathematicians, their geometry remains elusive. At criticality, they are conjectured to satisfy scaling limits towards Liouville quantum gravity surfaces with parameter $\gamma$, where $\gamma$ is related to the parameter $n \in (0,2]$ by $\gamma = 2\sqrt{1-\arccos(n/2)/\pi}$ or $\gamma = 2/\sqrt{1+\arccos(n/2)/\pi}$, decorated by an independent conformal loop ensemble of parameter $\kappa = 4/(1\pm \arccos(n/2)/\pi)$. See e.g.\@ Conjecture 2.1 in \cite{HL24} for a precise statement of the conjecture. In the nineties, $O(n)$ loop-decorated random planar maps were first studied in the physics literature using matrix model techniques, see e.g.\@ \cite{EK95,EK96} and the references therein. The precise model we deal with is the rigid $O(n)$ loop model on bipartite planar maps introduced in \cite{BBG12}, and extensively studied in \cite{BBG12bis,BBD16,B18,CCM20,AdSH24} using the gasket decomposition introduced in \cite{BBG12}. Actually, \cite{BBG12, BBG12bis, BBD16} do not restrict themselves at all to the rigid case. See also \cite{Korz22} for some results in the case where loops traverse triangles. However, the case $n=2$ was often excluded, except in \cite{EK96}, in Theorem 5 of \cite{B18}, and in a new version of \cite{AdSH24} which deals with the particular case of $O(2)$ loop-decorated quadrangulations and appeared the same day on the arXiv. At first sight, the case $n=2$ may be seen as a boundary case, but actually, in view of the scaling limit conjectures, it is the critical case since it is related to critical Liouville quantum gravity and to the conformal loop ensemble of parameter $\kappa =4$.

A planar map $\mathfrak{m}$ is a finite connected planar graph which is embedded in the sphere and seen up to orientation-preserving homeomorphism. Our planar maps are equipped with a distinguished oriented edge, called the root edge. The face on the right of the root edge is called the root face and is denoted by $f_r$. For every face $f$, we denote by $\deg (f)$ and we call degree of $f$ the number of edges surrounding $f$ with multiplicity. The degree of $f_r$ is called the perimeter of the map. We restrict ourselves to bipartite maps, i.e.\@ whose face degrees are even. Let $\mathcal{M}$ be the set of finite rooted bipartite planar maps and, for all $\ell \ge 1$, let $\mathcal{M}^{(\ell)}$ be the set of planar maps of perimeter $2\ell$. If $\mathfrak{m}$ is a map, we define the dual map $\mathfrak{m}^\dagger$ as the map obtained by exchanging the roles of the vertices and the faces. Two vertices of $\mathfrak{m}^\dagger$ are adjacent if the corresponding faces of $\mathfrak{m}$ share an edge on their boundary. A loop-decorated map $(\mathfrak{m},{\bf L})$ is a (finite rooted bipartite planar) map $\mathfrak{m}$ equipped with a loop configuration ${\bf L} = (\mathcal{L}_1, \ldots, \mathcal{L}_k)$ of disjoint unoriented simple closed paths on the dual map $\mathfrak{m}^\dagger$. Furthermore, the loops $\mathcal{L}_i$ do not go through the root face of $\mathfrak{m}$ and are rigid in the sense that they only visit quadrangles and they enter and exit quadrangles through opposite sides. See Figure \ref{carteO2}\footnote{All the planar maps drawn in this paper were obtained using the planar map editor of Timothy Budd.}. We denote the set of loop-decorated maps by $\mathcal{LM}$. For all $\ell \ge 1$, we also introduce the set $\mathcal{LM}^{(\ell)}$ of loop-decorated maps of perimeter $2\ell$.

Let ${\bf q} = (q_k)_{k\ge 1}$ be a non-zero sequence of non-negative real numbers. Define the weight of a planar map $\mathfrak{m}\in \mathcal{M}$ by
$$
w_{\bf q} (\mathfrak{m}) = \prod_{f \in \mathrm{Faces}(\mathfrak{m})\setminus\{f_r\} } q_{\deg(f)/2}.
$$
For all $\ell\ge 1$, we define the partition function of bipartite maps of weight sequence $\bf q$ and perimeter $2\ell$
$$
W^{(\ell)} = \sum_{\mathfrak{m} \in \mathcal{M}^{(\ell)}} w_{\bf q} (\mathfrak{m}).
$$
When $W^{(\ell)}<\infty$ for all $\ell\ge 1$, the sequence ${\bf q}$ is said admissible and we denote by $\P^{(\ell)}$ the associated Boltzmann probability measure characterized by $\P^{(\ell)}(\{\mathfrak{m}\}) = w_{\bf q}(\mathfrak{m})/W^{(\ell)}$ for all $\mathfrak{m} \in \mathcal{M}^{(\ell)}$. By Lemma 3.13 of \cite{StFlour}, when ${\bf q}$ is admissible, there exists a constant $c_{\bf q}$ such that
$$
\frac{W^{(\ell+1)}}{W^{(\ell)}} \mathop{\longrightarrow}\limits_{\ell \to \infty} c_{\bf q}.
$$
Similarly, let $\widetilde{\bf q} = (\tilde{q}_k)_{k\ge 1}$ be a sequence of non-negative real numbers and two real numbers $h,n\ge 0$. We define the weight of a loop decorated map $(\mathfrak{m}, {\bf L})\in \mathcal{LM}$ by setting
$$
w_{\widetilde{\bf q}, h,n} (\mathfrak{m}, {\bf L}) = \left(\prod_{\mathcal{L} \in {\bf L}} n h^{\vert \mathcal{L} \vert}\right) \left(\prod_{f \in \mathrm{Faces}(\mathfrak{m})\setminus(\{f_r\}\cup \bigcup_{\mathcal{L} \in {\bf L}} \mathcal{L})} \tilde{q}_{\deg(f)/2}\right),
$$
where $\vert \mathcal{L} \vert $ is the length of the loop $\mathcal{L}$, i.e.\@ its number of quadrangles. For all $\ell\ge 1$, we define the partition function of $O(n)$ loop-decorated maps with triple $(\widetilde{\bf q}, h,n)$ of perimeter $2 \ell$
$$
F^{(\ell)}(\widetilde{\bf q}, h,n) = \sum_{(\mathfrak{m}, {\bf L}) \in \mathcal{LM}^{(\ell)}} w_{\widetilde{\bf q}, h,n}(\mathfrak{m}, {\bf L}).
$$
We say that $(\widetilde{\bf q}, h,n)$ is admissible if $F^{(\ell)}(\widetilde{\bf q}, h, n)<\infty$ for all $\ell \ge 1$.

Let us now recall from \cite{BBG12} the gasket decomposition of a loop-decorated planar map. See also Sections 1.2.1 and 3.1 of \cite{B18} for a presentation of this decomposition. The gasket of a loop-decorated map $(\mathfrak{m}, {\bf L}) \in \mathcal{LM}$ is a planar map $\mathfrak{g} \in \mathcal{M}$ obtained by removing the interior of the outermost loops of $(\mathfrak{m}, {\bf L})$. See the right-hand side of Figure \ref{carteO2}. If we define the weight sequence ${\bf q}= (q_k)_{k\ge 1}$ by setting for all $k\ge 1$,
\begin{equation}\label{eq lien q qtilde}
	q_k = \tilde{q}_k + n h^{2k} F^{(k)}(\widetilde{\bf q}, h, n),
\end{equation}
then for all $\mathfrak{g}\in \mathcal{M}$,
$$
w_{\bf q}(\mathfrak{g})  = \sum_{\substack{(\mathfrak{m}, {\bf L}) \in \mathcal{LM}\\ \mathfrak{g} \text{ is the gasket of } (\mathfrak{m}, {\bf L})}}w_{\widetilde{\bf q}, h, n}(\mathfrak{m}, {\bf L}).
$$
In particular, summing over $(\mathfrak{m}, {\bf L})\in \mathcal{LM}^{(\ell)}$, we deduce that $W^{(\ell)} = F^{(\ell)}(\widetilde{\bf q}, h,n)$ for all $\ell \ge 1$, and that $\bf q$ is admissible if and only if the triple $(\widetilde{\bf q}, h,n)$ is admissible.

As in \cite{B18}, we say that the triple $(\widetilde{\bf q}, h,n)$ is non-generic critical if it is admissible, $n>0$ and $h = 1/c_{\bf q} $. Note that by \eqref{eq lien q qtilde}, a weight sequence ${\bf q} = (q_k)_{k\ge 1}$ is the weight sequence of the gasket of a non-generic critical $O(n)$ loop decorated map if and only if $\bf q$ is admissible and
\begin{equation}\label{eq gasket}
	\forall k\ge 1, \qquad \qquad q_k-n c_{\bf q}^{-2k} W^{(k)} \ge 0.
\end{equation}
Moreover, the associated triple $(\widetilde{\bf q}, h,n)$ of $O(n)$ loop-decorated maps is given by \eqref{eq lien q qtilde} and $h = 1/c_{\bf q}$. It is known since \cite{Kos88} that for $n \in (0,2)$, if the triple $(\widetilde{\bf q}, h,n)$ is non-generic critical, then we may have the following ``perimeter exponent'':
\begin{equation}\label{eq asymptotique fonction de partition}
	W^{(\ell)}\mathop{\sim}\limits_{\ell \to \infty} \frac{p_{\bf q}}{2} c_{\bf q}^{\ell+1} \ell^{-a}
\end{equation}
where $p_{\bf q}$ is some positive constant, $c_{\bf q} = 1/h$ and $a = 2\pm \arccos(n/2)/\pi$. See e.g.\@ Equations (15) and (16) of \cite{Kos88}, Equation (2.9) of \cite{BBD16} or Equation (3) of \cite{B18}. This perimeter exponent \eqref{eq asymptotique fonction de partition} is central in the study of random planar maps. Indeed, \eqref{eq asymptotique fonction de partition} means that $\bf q$ is non-generic critical of type $a$ in the sense of \cite{StFlour} (see Proposition 5.10 therein), or equivalently that the law of gasket $\P^{(\ell)}$ is the law of an $\alpha$-stable map with $\alpha = a-1/2$. The equivalent \eqref{eq asymptotique fonction de partition} actually determines the geometry of the gasket and thus many properties of critical $O(n)$ loop-decorated planar maps. The study of the geometry of $\alpha$-stable maps was initiated in \cite{LGM11} for the primal distances, then in \cite{BC,BCM} for the dual distances in the local limit and then in \cite{BBCK} for the dual distances to the root face under $\P^{(\ell)}$.

\subsection{Main results}

This work aims at establishing the asymptotic behaviour of the partition function, i.e.\@ the analogue of \eqref{eq asymptotique fonction de partition}, for critical $O(2)$ loop-decorated planar maps. Recall that a function $L :\R_+ \to \R$ is slowly varying if for all $\lambda>0$, the ratio $L(\lambda x)/L(x)$ converges to $1$ as $x\to \infty$.
\begin{theorem}\label{th carte o2}
	Let $\bf q$ be an admissible weight sequence satisfying \eqref{eq gasket} with $n=2$. Then there exists a slowly varying function $L_{\bf q}$ such that $1= O(L_{\bf q}(\ell))$, $L_{\bf q}(\ell) = O(\log \ell)$ and
	$$
	W^{(\ell)} \mathop{\sim}\limits_{\ell \to \infty } c_{\bf q}^{\ell+1} \frac{L_{\bf q}(\ell) }{2\ell^{2}}.
	$$
\end{theorem}
Note that, contrary to the equivalent \eqref{eq asymptotique fonction de partition} in the case $n\in (0,2)$, in the above equivalent, a slowly-varying function appears. See Section \ref{section exemples} for some examples where the slowly-varying function is constant or of order $\log \ell$. This result shows that a weight sequence $\bf q$ satisfying \eqref{eq gasket} with $n=2$ is critical non-generic of type $a=2$, not necessarily in the sense of \cite{StFlour}, but in the more general setting of \cite{Ric18,CR20} who allow slowly varying functions. See Subsection 2.1 of \cite{CR20} for the exact definition. Equivalently the weight sequence of the gasket of a critical $O(2)$ loop-decorated map is critical non-generic of type $a=2$, i.e.\@ the gasket is a $3/2$-stable map in the sense of \cite{Ric18,CR20}. 

Actually, we will prove an analogous result on a class of random walks which will imply the above theorem. Recall that a function $h: \Z \to \R_+$ is $\nu$-harmonic on a subset $A \subset \Z$ for a positive measure $\nu$ on $\Z$ if for all $p\in A$, we have $h(p) = \sum_{k \in \Z} \nu(k) h(p+k)$.  Let $h^\downarrow: \Z \to \Z$ be the function defined by 
\begin{equation}\label{eq h fleche vers le bas}
	\forall \ell \ge 0, \quad h^\downarrow(\ell)= 2^{-2\ell}  \binom{2\ell}{\ell} \qquad \text{and} \qquad
	\forall \ell \le -1, \quad h^\downarrow(\ell)=0.
\end{equation}
\begin{theorem}\label{th marche}
	Let $\nu$ be a probability distribution on $\Z$ such that $h^\downarrow$ is $\nu$-harmonic on $\Z_{\ge 1}$ and such that $\nu$ is not the Dirac mass at zero $\delta_0$. Assume that
	\begin{equation}\label{eq marche}
		\forall k\ge 1, \qquad \qquad \nu(k-1) \ge \nu(-k-1).
	\end{equation}
	Then,
	\begin{itemize}
		\item Either $\sum_{k\ge 1} k(\nu(k-1)-\nu(-k-1))\in [0,1)$ and there exists $c>0$ such that
		$$
		\nu(-k)\mathop{\sim}\limits_{k \to \infty} \frac{c \log k}{k^2};
		$$
		\item Or $\sum_{k\ge 1} k(\nu(k-1)-\nu(-k-1))=1$ and there exists a slowly varying function $L$ such that $L(k)= o(\log k)$, $1= O(L(k))$ and
		$$
		\nu(-k)\mathop{\sim}\limits_{k \to \infty} \frac{L(k)}{k^2}.
		$$
	\end{itemize}
\end{theorem}

\subsection{Applications and connections with other works}
Here, we mention some consequences of Theorem \ref{th carte o2} and we discuss connections with other works. 
In \cite{LGM11}, Le Gall and Miermont prove that $\alpha$-stable maps with $n$ vertices, when equipped with the primal graph distance, where each edge has length one, satisfy a scaling limit along subsequences towards a random metric space whose Hausdorff dimension is $2\alpha$. This result extends to the more general definition of $\alpha$-stable maps from \cite{Ric18,CR20} by adding a slowly varying function, and thus in particular to the gasket of critical (rigid) $O(2)$ loop-decorated maps.

Besides, one can also equip the gasket with some dual distances, where large faces become hubs so that the geometry lies in another universality class. In \cite{Kam23}, under $\P^{(\ell)}$ as $\ell\to \infty$, we establish the scaling limit of the first passage percolation distance $d^\dagger_{\mathrm{fpp}}$, obtained by putting i.i.d.\@ exponential random lengths of parameter $1$ on each dual edge, and the dual graph distance $d^\dagger_{\mathrm{gr}}$, obtained by assigning a length $1$ to each dual edge, between two uniform random vertices in a $3/2$-stable map and show that their diameter is of the same order. Here, the slowly varying function $L_{\bf q}$ significantly modifies the asymptotic behaviour of these distances.

In \cite{Kam23CLE}, we express the scaling limit of the distances $d^\dagger_\mathrm{gr}$ and $d^\dagger_\mathrm{fpp}$ from the faces of high degree to the root face. There is no doubt that the analogues of Theorem 1.3 and Proposition 6.6 in \cite{Kam23CLE} hold up to some slowly varying functions in the more general definition of $3/2$-stable maps from \cite{Ric18,CR20}. As a consequence, the results stated in Section 7.1 of \cite{Kam23CLE} are valid for any model of critical (rigid) $O(2)$ loop-decorated planar maps, up to some slowly varying functions. This gives a new piece of evidence that the scaling limit of $O(2)$ loop-decorated maps is a critical Liouville quantum gravity disk decorated with an independent conformal loop ensemble of parameter $\kappa =4$.

In \cite{B18}, Budd introduces a peeling exploration of $O(n)$ loop-decorated maps, which extends his peeling exploration of Boltzmann maps introduced in \cite{B16}, and obtains in Theorem 3 the scaling limit of the perimeter process as a positive self-similar Markov process. Thanks to our result, Theorem 3 of \cite{B18} has an analogue in the case $n=2$ for the pointed critical $O(2)$ loop-decorated map: one can see using Proposition 10 of \cite{B18} that the perimeter process $(P_i)_{i\ge 0}$ in Theorem 3 of \cite{B18} is the absolute value of a $\nu$-random walk stopped when it reaches zero (since the harmonic function $h^\downarrow_\mathfrak{p}$ defined in Equation (5) of \cite{B18} is constant) and the limiting process becomes the absolute value of a Cauchy process. In our case, a slowly varying function may appear.

Besides, the results of \cite{CCM20} describe the scaling limit of the sizes of the nested loops in terms of multiplicative cascades in the particular case of $O(n)$ loop-decorated quadrangulations and \cite{AdSH24} give the scaling limit of the volume of $O(n)$ loop-decorated quadrangulations. The case $n=2$ is treated by A\"id\'ekon, Da Silva and Hu in a new version of \cite{AdSH24} for $O(2)$ loop-decorated quadrangulations.

The dichotomy of Theorem \ref{th marche} appears in the particular case of $O(2)$ loop-decorated quadrangulations, i.e.\@ when $\widetilde{q}_k=0$ for all $k\neq 2$ and in the second case the slowly varying function is constant. More precisely, assume that $\bf q$ satisfies \eqref{eq asymptotique fonction de partition} with $n=2$ is obtained with a triple $(\widetilde{\bf q}, h,2)$ such that $\widetilde{q}_k=0$ for all $k\neq 2$. Then, as shown in Equation (1.7) of \cite{AdSH24}, 
\begin{itemize}
	\item Either $h>4/(3\pi^2)$, and the slowly varying function $L_{\bf q}$ is logarithmic, so that $W^{(\ell)} \sim  c_{\bf q}^{\ell+1} {c\log (\ell) }/({2\ell^{2}})$ as $\ell \to \infty$ for some constant $c>0$ (see also Subsection \ref{sous-section fully packed} for an example);
	\item Or $h=4/(3\pi^2)$ and the slowly varying function $L_{\bf q}$ is constant so that $W^{(\ell)} \sim  c_{\bf q}^{\ell+1} {c }/({2\ell^{2}})$ for some $c>0$.
\end{itemize}

The techniques of \cite{AdSH24} in the case of $O(2)$ loop-decorated quadrangulations, relying on approximation results as $n \uparrow 2$, were developed simultaneously with ours and are of independent interest. In view of Theorem \ref{th carte o2}, we conjecture that the analogous results hold for any model of (rigid) critical $O(2)$ loop-decorated maps. 

\subsection{Outline}
In Section \ref{section Wiener-Hopf}, we  describe our key tool: the Wiener--Hopf factorisation. In Section \ref{section preuves}, we prove Theorems \ref{th carte o2} and \ref{th marche}. Finally, in Section \ref{section exemples} we show that the computations in the proofs of Theorems \ref{th carte o2} and \ref{th marche} enable to characterise weight sequences of critical $O(2)$ loop-decorated planar maps and we give examples.

\section{Wiener--Hopf factorisation}\label{section Wiener-Hopf}
We recall here the Wiener--Hopf factorisation of a random walk on $\Z$. See e.g.\@ Section 2.3 of the lecture notes \cite{BuddPeeling} for more details on this tool and some applications to random planar maps, or Theorem 9.15 in the book \cite{Kal02}.

Let $\nu$ be a probability law on $\Z$. Let $(S_n)_{n\ge 0}$ be a random walk of step distribution $\nu$ starting at zero. The weak ascending ladder epochs $(T_i^\ge)_{i\ge 0}$ are the times at which $S_n$ reaches its running maximum. In other words, $T_0^\ge=0$ and for all $i \ge 0$,
$$
T_{i+1}^\ge = \inf\{n > T_i^\ge , \ S_n = \max_{0 \le k \le n} S_k\},
$$
where by convention $\inf \emptyset = \infty$. The weak ascending ladder heights $(H_i^\ge)_{i\ge 0}$ are defined by $H_i^\ge = S_{T_i^\ge}$ when $T_i^\ge <\infty$ and $H_i^\ge = \partial$ when $T_i^\ge = \infty$, where $\partial$ is a cemetery point. Similarly, the strict descending ladder epochs $(T^<_i)_{i\ge 0}$ are defined by $T^<_0= 0$ and for all $i\ge 0$, 
$$
T_{i+1}^< = \inf\{n > T_i^< , \ S_n < \min_{0 \le k \le n-1} S_k\}.
$$
The strict descending ladder heights $(H_i^<)_{i\ge 0}$ are defined by $H_i^< = S_{T_i^<}$ when $T_i^< <\infty$ and $H_i^< = \partial$ when $T_i^<= \infty$.

We denote the generating functions of $H^\ge_1$ and $H^<_1$ by
$$
G^\ge(z) = \E \left( z^{H_1^\ge}\right) = \sum_{k=0}^\infty z^k \P(H_1^\ge = k)
\quad \text{and} \quad
G^<(z) = \E \left( z^{H_1^<}\right) = \sum_{k=0}^\infty z^k \P(H_1^< = k)
$$
Note that when $S$ oscillates, $\P(H_1^\ge= \partial)= \P(H_1^<= \partial)=0$ so that $G^\ge(1)= G^<(1)=1$. We also denote by $\varphi$ the characteristic function of $S_1$ defined by $\varphi(\theta)= \E(e^{iS_1 \theta})$ for all $\theta \in \R$. The following proposition relates the functions $\varphi, G^\ge$ and $ G^<$.
\begin{proposition}\label{prop wiener hopf}(Wiener--Hopf factorisation)
	For all $\theta \in \R$, 
	$$
	1-\varphi(\theta) = (1-G^\ge(e^{i\theta}))(1-G^<(e^{-i\theta})).
	$$
\end{proposition}
The Wiener--Hopf factorisation was also used for the study of the peeling exploration of $O(n)$ loop-decorated planar maps by Timothy Budd in \cite{B18}.
\section{Proof of the main results}\label{section preuves}
\subsection{Proof of Theorem \ref{th carte o2} using Theorem \ref{th marche}}
\begin{proof}[Proof of Theorem \ref{th carte o2}]
	Let $\bf q$ be an admissible weight sequence satisfying \eqref{eq gasket}. By Lemma 5.2 of \cite{StFlour}, the measure $\nu$ on $\Z$ defined by setting for all $k\ge 0$,
	\begin{equation}\label{eq lien q nu}
		\nu(k) = q_{k+1} c_{\bf q}^k \qquad \text{and} \qquad \nu(-k-1 )  = 2 W^{(k)} c_{\bf q}^{-k-1}
	\end{equation}
	is a probability measure. Moreover, by Proposition 5.3 of \cite{StFlour}, the function $h^\downarrow$ is $\nu$-harmonic on $\Z_{\ge 1}$. Furthermore, \eqref{eq gasket} can be rewritten
	$$
	\forall k\ge 1, \qquad \nu(k-1) \ge \nu(-k-1).
	$$
	Finally, $\nu(-1)  = 2/c_{\bf q} \neq 0$ so that $\nu \neq \delta_0$. Thus, Theorem \ref{th marche} concludes the proof.
\end{proof}
\subsection{Proof of Theorem \ref{th marche}}
\begin{proof}[Proof of Theorem \ref{th marche}]
	It is well known by \cite{Don99} that, since $h^\downarrow$ is $\nu$-harmonic on $\Z_{\ge1}$ and $h^\downarrow(0)=1$, the function $h^\downarrow$ corresponds to the pre-renewal function of the random walk $S$, i.e.\@ for all $\ell \ge 0$, 
	$$
	h^\downarrow(\ell) 
	= \sum_{p\ge 0} \P(H^<_p= \ell),
	$$
	so that by the strong Markov property,
	$$
	\sum_{\ell \ge 0} h^\downarrow(\ell) z^\ell
	= \sum_{\ell \ge 0} \sum_{p \ge 0} \sum_{\substack{\ell_1, \ldots, \ell_p\ge 0 \\ \ell_1+\ldots +\ell_p=\ell}} \prod_{j=1}^p \left(\P(H^<_1= \ell_j)z^{\ell_j}\right)
	= \sum_{p\ge 0} G^<(z)^p= \frac{1}{1-G^<(z)}.
	$$
	Thus, in view of \eqref{eq h fleche vers le bas}, one computes
	\begin{equation}\label{eq G desc}
		G^<(z) = 1- \sqrt{1-z}.
	\end{equation}
	Besides, let us define the power series with non-negative coefficients
	$$
	g(z) = \sum_{k\ge 1} g_k z^k \coloneqq \sum_{k\ge 1} (\nu(k-1)-\nu(-k-1)) z^k.
	$$
	\emph{Step 1.} Our aim is to express $\nu(-k)$ using the $g_k$'s. Note that for all $\theta \in \R$,
	\begin{align*}
		\varphi(\theta) = &\sum_{k\le -2} \nu(k) e^{ik\theta} + \nu(-1)e^{-i\theta} + \sum_{k\ge 0}\nu(k) e^{ik\theta} \\
		&= e^{-i\theta}\nu(-1) + e^{-i\theta} g(e^{i\theta})
		+ e^{-i\theta} \sum_{k\ge 1} \nu(-k-1) (e^{ik\theta}+ e^{-ik\theta}).
	\end{align*}
	In particular, for all $\theta \in \R$,
	$$
	e^{i\theta}\varphi(\theta) - e^{-i\theta}\varphi(-\theta) = g(e^{i\theta}) - g(e^{-i\theta}).
	$$
	By taking \eqref{eq G desc} and Proposition \ref{prop wiener hopf} into account, we deduce that for all $\theta \in \R$,
	\begin{align}
		&g(e^{i\theta})-g(e^{-i\theta})\notag\\
		&= e^{i\theta}-e^{-i\theta} - e^{i\theta} \sqrt{1-e^{-i\theta}} \left(1-G^\ge(e^{i\theta})\right) + e^{-i\theta} \sqrt{1-e^{i\theta}} \left(1-G^\ge (e^{-i\theta})\right).\label{eq wiener hopf symetrise}
	\end{align}
	Let us define the power series
	\begin{equation}\label{eq def f}
		f(z) = \sum_{k\ge 1} f_k z^k \coloneqq \frac{z}{\sqrt{1-z} }\left(1-G^\ge(z)\right).
	\end{equation}
	We first compute the $f_k$'s and then express the $\nu(-k)$'s using the $f_k$'s. Note that $\theta\mapsto f(e^{i\theta}) \in \mathrm{L}^1([0,2\pi])$ and that \eqref{eq wiener hopf symetrise} can be rewritten as
	$$
	\frac{1}{\sqrt{1-e^{i\theta}}\sqrt{1-e^{-i\theta}}} \left(-g(e^{i\theta})+ g(e^{-i\theta})+ e^{i\theta}-e^{-i\theta}\right) = f(e^{i\theta})-f(e^{-i\theta}),
	$$
	hence, for all $k\ge 1$,
	\begin{align}
		f_k &= \frac{1}{2\pi} \int_0^{2\pi} \frac{1}{\sqrt{1-e^{i\theta}}\sqrt{1-e^{-i\theta}}} \left(-g(e^{i\theta})+ g(e^{-i\theta})+ e^{i\theta}-e^{-i\theta}\right)  e^{-ik\theta} d\theta \notag\\
		&=\frac{1}{2\pi} \int_0^{2\pi} \frac{1}{\vert 1-e^{i\theta} \vert} \left(2i\sin \theta - \sum_{j \ge 1} g_j 2 i \sin(j \theta) \right) e^{-ik\theta} d\theta \notag\\
		&= \frac{1}{2\pi} \int_0^{2\pi}  \frac{1}{ \sin (\theta/2)} \left( \sin \theta - \sum_{j\ge 1} g_j \sin(j \theta) \right) \sin (k\theta) d\theta.\label{eq f k}
	\end{align}
	Next, for all integers $k, j \ge 1$,
	\begin{align*}
		\int_0^{2 \pi} \frac{\sin(k\theta)}{\sin(\theta/2)} \sin(j\theta) d\theta%&=\int_0^{2 \pi}  \frac{e^{ik\theta}-e^{-ik\theta}}{e^{i\theta/2}+e^{-i\theta/2}} \sin(\ell \theta) d \theta\\
		&=\int_0^{2 \pi} \frac{e^{2ki\theta}-1}{e^{i\theta}-1} e^{i(1/2- k)\theta} \sin(j\theta) d\theta \\
		&= \int_0^{2\pi} \sum_{m=0}^{2k-1} e^{i(m+1/2-k)\theta} \sin(j \theta) d \theta \\
		&= 2\sum_{m=0}^{k-1} \int_0^{2\pi} \cos((m+1/2)\theta) \sin(j\theta) d \theta\\
		&=2\sum_{m=0}^{k-1} \int_0^{2\pi} \frac{1}{2}\left( \sin((m+j+1/2)\theta) - \sin((m-j+1/2)\theta) \right) d\theta\\
		&= 2 \sum_{m=0}^{k-1} \left( \frac{1}{m+j+1/2} - \frac{1}{m-j+1/2} \right) \\
		&= 2 \sum_{m=j}^{k+j-1} \frac{1}{m+1/2} - 2 \sum_{m=-j}^{k-j-1} \frac{1}{m+1/2}\\
		&=2 \sum_{m=k-j}^{k+j-1} \frac{1}{m+1/2}.
	\end{align*}
	Therefore, \eqref{eq f k} and the fact that $\sum_{j\ge 1} \vert g_j \vert \le 1< \infty$ yield, by Fubini-Tonelli, for all $k\ge 1$,
	\begin{equation}\label{eq f k 2}
		f_k = \frac{1}{\pi}\left( \frac{1}{k-1/2} + \frac{1}{k+1/2}\right) - \frac{1}{\pi} \sum_{j \ge 1} g_j\sum_{m=k-j}^{k+j-1} \frac{1}{m+1/2}.
	\end{equation}
	Then, let us compute $\nu(-k)$ for $k\ge 1$ in terms of the $f_k$'s. By Proposition \ref{prop wiener hopf}, \eqref{eq G desc} and \eqref{eq def f}, for all $\theta \in (0,2 \pi)$,
	\begin{equation}\label{eq lien phi f}
		1- \varphi(\theta) = \sqrt{1- e^{-i\theta}} \left( 1- G^\ge (e^{i\theta})\right) = \frac{\sqrt{1-e^{-i\theta}}\sqrt{1-e^{i\theta}}}{e^{i\theta}} f(e^{i\theta}) = \frac{2 \vert \sin(\theta/2) \vert}{e^{i\theta}} f(e^{i\theta}).
	\end{equation}
	Consequently, assuming that $\theta \mapsto f(e^{i\theta}) \in \mathrm{L}^2([0,2\pi])$ for all $k \in \Z $, 
	\begin{align*}
		\nu(k) &={\bf 1}_{k=0}- \frac{1}{2\pi} \int_0^{2 \pi} \frac{2 \sin(\theta/2)}{e^{i\theta}} \sum_{\ell \ge 1} f_\ell e^{ i\ell \theta} e^{-ik \theta} d \theta%=-\frac{1}{\pi} \int_0^{2 \pi} \sum_{\ell \ge 1} f_\ell \sin(\theta/2) e^{i(\ell-k-1)\theta} d\theta
		\\
		&= {\bf 1}_{k=0}-\frac{1}{\pi} \sum_{\ell \ge 1}  \int_0^{2 \pi} f_\ell \sin(\theta/2) \cos((\ell-k-1)\theta)d\theta
		\\
		&={\bf 1}_{k=0}-\frac{1}{2\pi} \sum_{\ell \ge 1} \int_0^{2 \pi}  f_\ell\left(\sin((\ell-k-1/2)\theta) - \sin((\ell+k- 3/2)\theta)\right)d\theta
		\\
		&={\bf 1}_{k=0}+\frac{1}{\pi} \sum_{\ell \ge 1} f_{\ell}  \left( \frac{1}{\ell + k - 3/2} - \frac{1}{\ell-k-1/2}\right),
	\end{align*}
	where in the second equality we use the assumption that $\theta \mapsto  f(e^{i\theta}) \in \mathrm{L}^2([0,2\pi])$ in order to exchange the series and the integral. Thus, for all $k\in \Z $, 
	\begin{equation}\label{eq nu k}
		\nu(k) = {\bf 1}_{k=0}+\frac{1}{\pi} \sum_{\ell \ge 1} f_\ell \frac{4}{4(\ell-k-1)^2-1}.
	\end{equation}
	The above equality is also valid when we do not assume that $\theta \mapsto f(e^{i\theta}) \in \mathrm{L}^2([0,2\pi])$ by density of $\mathrm{L}^2([0,2\pi])$ in $\mathrm{L}^1([0,2\pi])$, since $\nu(k)  = {\bf 1}_{k=0}  - (1/(2\pi))\int_0^{2\pi} (2 \vert \sin (\theta/2)\vert)e^{-i\theta} f(e^{i\theta}) e^{-ik\theta}d\theta$ and $f_\ell = (1/(2\pi))\int_0^{2\pi} f(e^{i\theta}) e^{-i\ell \theta} d\theta$.
	In view of \eqref{eq f k 2}, we have completed the first step.
	
	\emph{Step 2.} Let us now prove that $\sum_{j \ge 1} j g_j \le1$. This will be useful in order to obtain the asymptotics of $f_k$ as $k \to \infty$. Assume by contradiction that $\sum_{j \ge 1} jg_j \in (1, \infty]$. Then, by \eqref{eq f k 2},
	$$
	\forall \ell \ge 1, \qquad \qquad f_\ell\le \frac{2}{\pi} \left( \frac{1}{\ell-1/2} - \sum_{j\ge 1} jg_j \frac{1}{j+\ell-1/2} {\bf 1}_{j\le \ell}\right),
	$$
	so that there exist $\vp >0$ and $\ell_0\ge 1$ such that for all $\ell\ge \ell_0$, $f_\ell \le -\vp/\ell$. So
	\begin{align*}
		&\sum_{\ell \ge \ell_0} f_\ell \frac{4}{4(\ell+k-1)^2-1}\le -\vp \sum_{\ell \ge \ell_0} \frac{1}{\ell}  \frac{4}{4(\ell+k-1)^2-1}\\
		&\le -\vp \int_{\ell_0}^\infty \frac{1}{x}  \frac{4}{4(x+k-1)^2-1}dx\sim -\vp\frac{\log k}{k^2}
	\end{align*}
	as $k \to \infty$, where the equivalent is obtained after taking the change of variable $y=x/k$. But the sum of the first $\ell_0$ terms is a $O(1/k^2)$, so that by \eqref{eq nu k} for all $k$ large enough, $\nu(-k)<0$, absurd. As a result, $\sum_{j \ge 1} j g_j\le 1$. 
	
	\emph{Step 3.} Let us show that
	\begin{equation}\label{eq sommable}
		\sum_{k \ge 1} \sum_{j \ge k/2} g_j \sum_{m=k-j}^{k+j-1} \frac{1}{m+1/2} <\infty.
	\end{equation}
	One can first see that for all $k \ge1$,
	$$
	\sum_{j \ge 2k} g_j \sum_{m=k-j}^{k+j-1} \frac{1}{m+1/2}= 
	\sum_{j \ge 2k} g_j \sum_{m=j-k}^{k+j-1} \frac{1}{m+1/2} 
	\le  \sum_{j \ge 2k} g_j 2k \frac{1}{k}
	\le 2\sum_{j \ge k} g_j 
	$$
	and $\sum_{k\ge 1} \sum_{j \ge k} g_j =\sum_{j\ge 1} j g_j <\infty$. 
	
	Moreover,
	\begin{align*}
		\sum_{k\ge 1} \sum_{j=k}^{2k}g_j \sum_{m=k-j}^{k+j-1} \frac{1}{m+1/2} &=
		\sum_{k\ge 1} \sum_{j=k}^{2k}g_j \sum_{m=j-k}^{k+j-1} \frac{1}{m+1/2}\\
		&\le C \sum_{k\ge 1} \sum_{j=k}^{2k}g_j \log \left( \frac{j+ k+ 1/2}{j-k+1/2}\right)\\
		&=C \sum_{j \ge 1} \sum_{j/2 \le k \le j} g_j \log \left( \frac{j+ k+ 1/2}{j-k+1/2}\right)\\
		&\le C' \sum_{j\ge 1} g_j \int_{j/2}^j \log \left( \frac{j+x+1/2}{j-x+1/2}\right) dx,
		%&= C' \sum_{j\ge 1} g_j \left( \left(2j+\frac{1}{2}\right) \log\left(2j+\frac{1}{2}\right)  - \left(2j+\frac{1}{2}\right)  
		%- \left(\frac{3}{2}j+\frac{1}{2}\right) \log \left(\frac{3}{2}j+\frac{1}{2}\right) + \left(\frac{3}{2}j+\frac{1}{2}\right) 
		%- \left(\frac{j}{2}+\frac{1}{2}\right) \log \left(\frac{j}{2}+\frac{1}{2}\right) + \left(\frac{j}{2}+\frac{1}{2}\right)
		%+ \frac{1}{2}\log \frac{1}{2}- \frac{1}{2}      \right)
	\end{align*}
	where $C,C'$ are positive constants. The sum in the last line is finite since for all $j\ge 1$,
	\begin{align*}
		\int_{j/2}^j \log &\left( \frac{j+x+1/2}{j-x+1/2}\right) dx
		\\=&\left(2j+\frac{1}{2}\right) \log\left(2j+\frac{1}{2}\right)  - \left(2j+\frac{1}{2}\right)  
		- \left(\frac{3}{2}j+\frac{1}{2}\right) \log \left(\frac{3}{2}j+\frac{1}{2}\right)\\ 
		&+ \left(\frac{3}{2}j+\frac{1}{2}\right) 
		- \left(\frac{j}{2}+\frac{1}{2}\right) \log \left(\frac{j}{2}+\frac{1}{2}\right) + \left(\frac{j}{2}+\frac{1}{2}\right)
		+ \frac{1}{2}\log \frac{1}{2}- \frac{1}{2}   \\
		\le &C''j,
	\end{align*}
	for some constant $C''>0$ and since $\sum_{j\ge 1} jg_j <\infty$.

	Similarly, for some constants $C,C'>0$,
	\begin{align*}
		\sum_{k\ge 1} \sum_{k/2\le j \le k}g_j \sum_{m=k-j}^{k+j-1} \frac{1}{m+1/2} 
		&\le C \sum_{k\ge 1} \sum_{k/2 \le j \le k}g_j \log \left( \frac{j+ k+ 1/2}{k-j+1/2}\right)\\
		&=C \sum_{j \ge 1} \sum_{j \le k \le 2j} g_j \log \left( \frac{j+ k+ 1/2}{k-j+1/2}\right)\\
		&\le C' \sum_{j\ge 1} g_j \int_{j}^{2j} \log \left( \frac{j+x+1/2}{x-j+1/2}\right) dx<\infty.
	\end{align*}
	This ends the proof of \eqref{eq sommable}.
	
	\emph{Step 4.} Next, we show that
	\begin{equation}\label{eq sommable2}
		\sum_{\ell \ge 1} \left\vert f_\ell - \frac{2(1- \sum_{j= 1}^{\ell/2} jg_j)}{\pi \ell}  \right\vert<\infty.
	\end{equation}
	Note that by \eqref{eq f k 2},
	\begin{align*}
		&\left\vert f_\ell - \frac{2(1- \sum_{j= 1}^{\ell/2} jg_j)}{\pi \ell}  \right\vert\le \frac{1}{\pi}\left\vert \frac{1}{\ell-1/2} + \frac{1}{\ell+1/2} - \frac{2}{\ell}\right\vert\\
		&{} \qquad \qquad \qquad+\frac{1}{\pi} \sum_{1 \le j \le \ell/2} g_j \left\vert \frac{2j}{\ell}-\sum_{m=\ell-j}^{\ell+j-1}\frac{1}{m+1/2}\right\vert
		+ \frac{1}{\pi} \sum_{j\ge  \ell/2} g_j \sum_{m=\ell-j}^{\ell+j-1} \frac{1}{m+1/2}.
	\end{align*}
	The first term is clearly summable. The last one is summable by \eqref{eq sommable}. Finally, for the second term, we upperbound
	\begin{align*}
		\sum_{\ell \ge 1}\sum_{1 \le j \le  \ell/2} g_j \left\vert \frac{2j}{\ell}-\sum_{m=\ell-j}^{\ell+j-1}\frac{1}{m+1/2}\right\vert 
		&=
		\sum_{\ell \ge 1}\sum_{1 \le j \le \ell/2} g_j \left\vert \sum_{m=\ell-j}^{\ell+j-1}\frac{m+1/2-\ell}{\ell(m+1/2)}\right\vert \\
		&\le C'' \sum_{\ell \ge 1} \sum_{1 \le j \le \ell/2} g_j \frac{j^2}{\ell^2},
	\end{align*}
	where $C''>0$ is a constant, and
	$$
	\sum_{\ell \ge 1} \sum_{1 \le j \le \ell/2} g_j \frac{j^2}{\ell^2} = \sum_{j\ge 1} g_j j^2\sum_{\ell\ge 2 j} \frac{1}{\ell^2} \le \sum_{j\ge 1} g_j j^2 \frac{1}{j} <\infty.
	$$
	This proves \eqref{eq sommable2}.
	
	\emph{Step 5.} Let us now deal with the case $\sum_{j \ge 1} j g_j \in [0,1)$ so as to prove the first point of Theorem \ref{th marche}. By \eqref{eq nu k}, for all $k\ge1$,
	\begin{align*}
		\nu(-k)= 
		&\frac{1}{\pi} \sum_{\ell \ge 1} \left(f_\ell - \frac{2(1-\sum_{j= 1}^{\ell/2} jg_j)}{\pi \ell} \right)
		\frac{4}{4(\ell+k-1)^2-1} \\
		&+ \sum_{\ell \ge 1} \frac{2(1- \sum_{j=1}^{\ell/2} jg_j)}{\pi^2} \frac{1}{\ell} \frac{4}{4(\ell+k-1)^2-1} .
	\end{align*}
	The first term is a $O(1/k^2)$ as $k \to \infty$ by \eqref{eq sommable2}. For the second term, for all $\ell_0\ge 1$,
	\begin{align*}
		\int_{\ell_0}^\infty \frac{1}{x} \frac{4}{4(x+k-1)^2-1} dx &\le \sum_{\ell \ge \ell_0} \frac{1}{\ell} \frac{4}{4(\ell+k-1)^2-1} \\
		&\le \frac{4}{4k^2-1} + \int_{\ell_0}^\infty \frac{1}{x} \frac{4}{4(x+k-1)^2-1} dx.
	\end{align*}
	Moreover, after the change of variable $y= x/k$, the integral can be bounded by
	\begin{align*}
		\frac{1}{k^2} \int_{\ell_0/k}^\infty \frac{1}{y} \frac{4}{4(y+1)^2} dy &\le \frac{1}{k^2}\int_{\ell_0/k}^\infty \frac{1}{y} \frac{4}{4(y+1-1/k)^2-1/k^2}  dy\\
		&\le \frac{1}{(1-1/k)^3}\frac{1}{k^2} \int_{\ell_0/k}^\infty \frac{1}{y} \frac{4}{4(y+1)^2} dy.
	\end{align*}
	Hence, since $\ell_0$ can be chosen arbitrarily large so that $\sum_{j=1}^{\ell_0/2} jg_j$ is close to $\sum_{j\ge 1} j g_j$,
	\begin{equation}\label{eq equivalent log k sur k carre}
		\nu(-k) \mathop{\sim}\limits_{k\to \infty} \frac{2(1- \sum_{j\ge 1} jg_j)}{\pi^2} \frac{\log k}{k^2}.
	\end{equation}
	
	\emph{Step 6.} We finally deal with the case $\sum_{j\ge 1} j g_j = 1$ to prove the second point of Theorem \ref{th marche}. We first assume that 
	$$
	\sum_{\ell \ge 1} \vert f_\ell \vert =\infty,
	$$
	which is equivalent by \eqref{eq sommable2} to the fact that 
	\begin{equation}\label{eq f ell pas sommable}
		\sum_{\ell \ge 1}  \frac{1- \sum_{j=1}^{\ell/2} j g_j }{\ell} = \infty.
	\end{equation}
	We define the function $L$ by setting
	$$
	\forall x\ge 0, \qquad  L(x) = \frac{1}{\pi} \sum_{\ell \ge 1} f_\ell \frac{4 x^2}{4(\ell+x-1)^2-1}.
	$$
	Let us show that $L$ is a slowly varying function in the sense that for all $\lambda>0$, the ratio $L(\lambda x)/L(x)$ converges to $1$ as $x\to \infty$. Let us also introduce the function $\widetilde{L} $ defined for all $x \ge 0$ by
	$$
	\widetilde{L}(x) = \frac{2}{\pi^2}  \sum_{\ell \ge 1} 
	\frac{1- \sum_{j=1}^{\ell/2} j g_j }{\ell} \frac{4 x^2}{4(\ell+x-1)^2-1}.
	$$
	By \eqref{eq sommable2} and since for all $x\ge 2, \ell\ge1$, we have $4x^2/(4(\ell+x-1)^2-1)\le 2$, 
	\begin{equation}\label{eq L L tilde}
		L(x) = \widetilde{L}(x) + O(1)
	\end{equation} 
	as $x \to \infty$.
	
	Let $\lambda >1 $. One upperbounds for all $x$ large enough,
	\begin{align*}
		&\big\vert\widetilde{L}(\lambda x) - \widetilde{L}(x) \big\vert \\
		&\le \frac{2}{\pi^2} \sum_{\ell \ge 1} \frac{1- \sum_{j=1}^{\ell/2} j g_j }{\ell} 
		\left\vert
		\frac{4\lambda^2x^2(4(\ell+x-1)^2-1) - 4x^2(4(\ell+\lambda x -1)^2-1)}{(4(\ell+\lambda x -1)^2-1)(4(\ell+x-1)^2-1)}
		\right\vert \\
		&= \frac{2}{\pi^2} \sum_{\ell \ge 1} \frac{1- \sum_{j=1}^{\ell/2} j g_j }{\ell} \\ 
		&\quad \times\left\vert \frac{(4\lambda x (\ell+x-1) - 4x(\ell + \lambda x-1))(4\lambda x(\ell+x-1) + 4x(\ell+\lambda x -1)) - 4 (\lambda^2-1)x^2}{(4(\ell+\lambda x -1)^2-1)(4(\ell+x-1)^2-1)}\right\vert\\
		&=\frac{2}{\pi^2}  \sum_{\ell \ge 1} \frac{1- \sum_{j=1}^{\ell/2} j g_j }{\ell} 
		\left\vert \frac{4(\lambda-1)(\ell-1)x(4\lambda x (\ell+x-1)+ 4 x(\ell+ \lambda x -1)) - 4(\lambda^2-1) x^2}{(4(\ell+\lambda x -1)^2-1)(4(\ell+x-1)^2-1)}\right\vert\\
		&\le C(\lambda)\sum_{\ell \ge 1} \frac{1}{\ell} \frac{x^2 \ell(\ell+x)}{(\ell + x)^4} \\
		&\le C(\lambda) \int_0^\infty \frac{x^2 }{(x+y)^3} dy\\
		&= C(\lambda) \int_0^\infty \frac{1}{(1+z)^3} dz,
	\end{align*}
	where $C(\lambda)$ is a positive constant depending on $\lambda$.
	In particular, 
	$$
	\widetilde{L}(\lambda x) - \widetilde{L}(x) = O(1)
	$$
	as $x\to \infty$. But by \eqref{eq f ell pas sommable}, we know that $\widetilde{L}(x) \to \infty$ as $x\to \infty$. As a consequence, $\widetilde{L}$ is slowly varying. By \eqref{eq L L tilde}, $L$ is therefore slowly varying and $L(x) \to \infty$ as $x \to \infty$. Besides, since $\sum_{j\ge 1} j g_j = 1$, one can see that $\widetilde{L}(x) = o(\log x)$ as $x \to \infty$ by a comparison between series and integrals, hence $L(x) = o(\log (x))$ as $x \to \infty$ by \eqref{eq L L tilde}. This ends the proof of the second point in the case where $\sum_{\ell\ge 1} \vert f_\ell \vert = \infty$.
	
	\emph{Step 7.}   In order to conclude, it remains to deal with the case where $\sum_{j\ge 1} j g_j = 1$ and
	\begin{equation}\label{eq f ell sommable}
		\sum_{\ell \ge 1} \vert f_\ell \vert <\infty 
	\end{equation}
	and to show that $\sum_{\ell\ge 1} f_\ell \neq 0$. Indeed, assuming \eqref{eq f ell sommable}, by dominated convergence and \eqref{eq nu k}, we get
	$$
	k^2 \nu(-k) \mathop{\longrightarrow}\limits_{k\to \infty} \frac{1}{\pi} \sum_{\ell \ge 1} f_\ell.
	$$
	By \eqref{eq sommable2}, \eqref{eq f ell sommable} and the fact that $\sum_{j\ge 1} j g_j = 1$, we have $\sum_{\ell\ge 1} (1/\ell)\sum_{j> \ell/2} jg_j <\infty$, hence by Fubini-Tonelli,
	\begin{equation}\label{eq g j j log j}
		\sum_{j\ge 1} g_j j \log j <\infty.
	\end{equation}
	Next, by \eqref{eq f k 2} and since $\sum_{j\ge 1} jg_j = 1$, for all $n\ge 1$,
	\begin{align*}
		\sum_{\ell=1}^n f_\ell &= \frac{1}{\pi} \sum_{j\ge 1} g_j \sum_{\ell=1}^n
		\left( \frac{j}{\ell-1/2} + \frac{j}{\ell+1/2} - \sum_{m=\ell-j}^{\ell+j-1} \frac{1}{m+1/2}\right),
	\end{align*}
	and for all $j\ge 1$,
	\begin{align*}
		\sum_{\ell=1}^n
		&\left( \frac{j}{\ell-1/2} + \frac{j}{\ell+1/2} - \sum_{m=\ell-j}^{\ell+j-1} \frac{1}{m+1/2}\right)\\
		&= \sum_{\ell=1}^n \sum_{m=0}^{j-1}\left(\frac{1}{\ell-1/2} + \frac{1}{\ell+1/2} - \frac{1}{\ell+m+1/2}-\frac{1}{\ell-m-1/2} \right) \\
		&= \sum_{m=0}^{j-1} \left(\sum_{\ell=1}^n\left(\frac{1}{\ell-1/2} + \frac{1}{\ell+1/2} \right)- 
		\sum_{\ell=1+m}^{n+m} \frac{1}{\ell+1/2} - \sum_{\ell=1-m}^{n-m} \frac{1}{\ell-1/2} 
		\right)\\
		&=\sum_{m=0}^{j-1}\left( \sum_{\ell=1}^m \frac{1}{\ell+1/2} - \sum_{\ell= n+1}^{n+m} \frac{1}{\ell+1/2} + \sum_{\ell=n-m+1}^n \frac{1}{\ell-1/2} - \sum_{\ell=1-m}^0 \frac{1}{\ell-1/2}\right)\\
		&= \sum_{m=0}^{j-1} \left(\sum_{\ell=1}^m \left(\frac{1}{\ell+1/2} + \frac{1}{\ell-1/2}\right)+ \sum_{\ell=1}^m \left(\frac{1}{n-m+\ell-1/2} - \frac{1}{n+\ell+1/2}\right)\right)\\
		&= \sum_{m=0}^{j-1} \sum_{\ell=1}^m \left(\frac{1}{\ell+1/2} + \frac{1}{\ell-1/2}\right)+ 	\sum_{m=0}^{j-1}\sum_{\ell=1}^m\frac{m+1}{(n+\ell+1/2)(n-m+\ell-1/2)}.
	\end{align*}
	Note that for all $n\ge1$, $m\ge0$ and $\ell\in \lb 1, m \rb$,
	$$
	\left\vert\frac{m+1}{(n+\ell+1/2)(n-m+\ell-1/2)}\right\vert\le \left\{\begin{matrix}
		2 /\vert{n-m+\ell-1/2}\vert & \text{ if }& n\ge (m+1)/2 \\
		4/(n+\ell+1/2)& \text{ if }& n\le m/2 \text{ and } \ell \le (m+1)/4\\
		4/\vert{n-m+\ell-1/2}\vert& \text{ if }& n\le m/2 \text{ and } \ell \ge (m+1)/4.
	\end{matrix}\right.
	$$
	As a result, there exists $C>0$ such that for all $j\ge 1$,
	$$
	\left\vert\sum_{m=0}^{j-1}\sum_{\ell=1}^m\frac{m+1}{(n+\ell+1/2)(n-m+\ell-1/2)}\right\vert\le
	C\sum_{m=1}^{j} \log m \le C(j+1)\log (j+1).
	$$
	Thus, in view of \eqref{eq g j j log j}, by dominated convergence,
	$$
	\sum_{\ell=1}^n f_\ell \mathop{\longrightarrow}\limits_{n\to \infty} \frac{1}{\pi} \sum_{j\ge 1} g_j \sum_{m=0}^{j-1} \sum_{\ell=0}^m \left(\frac{1}{\ell+1/2} +\frac{1}{\ell-1/2}\right),
	$$
	and the limit is non-zero since $\sum_{j\ge 1} jg_j = 1$ so that there exists $j\ge 1$ such that $g_j>0$ and we cannot have $g_1=1$ since $\nu \neq \delta_0$, hence there exists $j\ge 2$ such that $g_j>0$. This ends the proof.
\end{proof}

\section{Weight sequences}\label{section exemples}
In this section, we show how the computations in the proof of Theorem \ref{th marche} enable to characterise weight sequences of critical $O(2)$ loop-decorated maps and we give two examples.
\subsection{Characterisation of weight sequences of critical $O(2)$ decorated maps}
Note that in the proof of Theorem \ref{th marche} we computed $\nu(k)$ for all $k \in \Z$ in terms of the $g_j = \nu(j-1)-\nu(-j-1)$ for $j\ge 1$. Indeed, combining \eqref{eq f k 2} and \eqref{eq nu k}, one gets for all $k \in \Z$,
\begin{align}
	&\nu(k)  =  {\bf 1}_{k=0}\notag
	\\&+\frac{1}{\pi} \sum_{\ell \ge 1}\left( \frac{1}{\pi}\Big(\frac{1}{\ell-1/2}+ \frac{1}{\ell+1/2}  - \sum_{j\ge 1} g_j \sum_{m=\ell-j}^{\ell+j-1} \frac{1}{m+1/2} \Big)\right) \frac{4}{4(\ell-k-1)^2-1}.\label{eq nu 2}
\end{align}
In order to obtain examples of weight sequences of critical $O(2)$-decorated planar maps, an approach can be to choose a sequence $(g_k)_{k\ge 1}$ of non-negative real numbers such that $\sum_{j\ge 1} j g_j \le 1$ and to see whether $\nu$ defined by \eqref{eq nu 2} is indeed a probability measure on $\Z$ such that $h^\downarrow$ is $\nu$-harmonic and $\nu \neq \delta_0$. The following proposition actually states that it suffices to check that $\nu(k) \ge 0$ for all $k\in \Z$.

\begin{proposition}\label{proposition}
	Let $(g_j)_{j\ge 1}$ be a sequence of non-negative real numbers such that $\sum_{j\ge 1} jg_j \le 1$. Let $\nu$ be defined by \eqref{eq nu 2}. Assume that $\nu(0)<1$ and that for all $k\in \Z$, we have $\nu(k)\ge 0$. Then $\nu$ is a probability distribution and $h^\downarrow$ is $\nu$-harmonic on $\Z_{\ge 1}$. Moreover, for all $k\ge 1$,
	$$
	\nu(k-1) - \nu(-k-1)= g_k\ge 0.
	$$
	In particular, the sequence ${\bf q}= (q_k)_{k\ge 1}$ defined by setting for all $k\ge 1$,
	$$
	q_k = \nu(k-1) (\nu(-1)/2)^{k-1}
	$$
	is the weight sequence of the gasket of a critical $O(2)$ loop-decorated map. The associated triple $(\widetilde{\bf q}, 1/c_{\bf q}, 2)$ is given by $c_{\bf q} = 2/\nu(-1)$ and for all $k\ge 1$,
	$$
	\tilde{q}_k = g_k (\nu(-1)/2)^{k-1}.
	$$
\end{proposition}

\begin{proof}
	As in the previous section, let $f_\ell$ be the expression in the big parenthesis in \eqref{eq nu 2}. Note that the steps 3 and 4 of the proof of Theorem \ref{th marche} only use the fact that $g_j\ge 0$ for all $j\ge 1$ and that $\sum_{j\ge 1} j g_j <\infty$ so that \eqref{eq sommable2} holds. As a result, after writing
	\begin{align*}
		\nu(k)= &{\bf 1}_{k=0} + \frac{1}{\pi} \sum_{\ell \ge 1} \left(f_\ell - \frac{2(1-\sum_{j= 1}^{\ell/2} jg_j)}{\pi \ell} \right)
		\frac{4}{4(\ell-k-1)^2-1}\\ &+ \sum_{\ell \ge 1} \frac{2(1- \sum_{j=1}^{\ell/2} jg_j)}{\pi^2} \frac{1}{\ell} \frac{4}{4(\ell-k-1)^2-1} ,
	\end{align*}
	one can apply Fubini-Tonelli theorem to obtain that
	\begin{align*}
		\sum_{k\in \Z} \nu(k) = 1 +  &\frac{1}{\pi} \sum_{\ell \ge 1} \left(f_\ell - \frac{2(1-\sum_{j= 1}^{\ell/2} jg_j)}{\pi \ell} \right)
		\sum_{k \in \Z}\frac{4}{4(\ell-k-1)^2-1} \\
		&+ \sum_{\ell \ge 1} \frac{2(1- \sum_{j=1}^{\ell/2} jg_j)}{\pi^2} \frac{1}{\ell} \sum_{k \in \Z}\frac{4}{4(\ell-k-1)^2-1}.
	\end{align*}
	Moreover, after identifying a telescopic sum,
	$$
	\sum_{k \in \Z}\frac{4}{4(\ell-k-1)^2-1} = 2\sum_{k\ge 1} \frac{4}{4k^2-1} - 4= 2\times 2 -4  = 0.
	$$
	This proves that $\nu$ is a probability measure. 
	
	Next, let us check that the function $h^\downarrow$ defined in \eqref{eq h fleche vers le bas} is $\nu$-harmonic on $\Z_{\ge 1}$. Let $p\ge 1 $. Since $h^\downarrow$ is non-negative and is bounded, we may apply again Fubini-Tonelli and obtain that 
	\begin{align*}
		\sum_{k\in \Z} h^\downarrow(p+k)\nu(k) =&h^\downarrow(p)\\
		&+  \frac{1}{\pi} \sum_{\ell \ge 1} \left(f_\ell - \frac{2(1-\sum_{j= 1}^{\ell/2} jg_j)}{\pi \ell} \right)
		\sum_{k \in \Z}h^\downarrow(p+k)\frac{4}{4(\ell-k-1)^2-1} \\
		&+ \sum_{\ell \ge 1} \frac{2(1- \sum_{j=1}^{\ell/2} jg_j)}{\pi^2} \frac{1}{\ell} \sum_{k \in \Z} h^\downarrow(p+k)\frac{4}{4(\ell-k-1)^2-1}.
	\end{align*}
	So, in order to prove that $h^\downarrow $ is $\nu$-harmonic, it suffices to check that for all $p\ge 1 $,
	\begin{equation}\label{eq h fleche vers le bas harmonique}
		\sum_{k \in \Z} h^\downarrow(p+k)\frac{4}{4(\ell-k-1)^2-1} = 0.
	\end{equation}
	Let $\nu_{\mathrm{sym}}$ be the probability distribution defined by
	$$
	\forall k \in \Z, \qquad \nu_{\mathrm{sym}} (k)  = {\bf 1}_{k=0} + \frac{2}{\pi} \frac{1}{4k^2-1}.
	$$
	One can check that it is indeed a probability distribution using a telescopic sum. This probability distribution appears in a particular weight sequence of $O(2)$ loop-decorated maps introduced in Remark 1 of \cite{B18}. Its characteristic function is given for all $\theta \in \R$ by
	\begin{equation}\label{fonction caracteristique sym}
		\varphi_{\mathrm{sym}}(\theta) = 1 - \vert \sin (\theta/2)\vert.
	\end{equation}
	Indeed, one readily computes for all $k \in \Z$,
	\begin{align*}
		\frac{1}{2\pi}\int_0^{2\pi} \sin(\theta/2) e^{-ik\theta} d \theta% = \frac{1}{2\pi}\int_0^{2\pi} \sin(\theta/2) \cos(k\theta) d \theta 
		&= \frac{1}{4\pi}\int_0^{2\pi} (\sin((k+1/2)\theta)-\sin((k-1/2)\theta)) d\theta \\
		&= \frac{1}{4\pi}\left(\frac{2}{k+1/2} - \frac{2}{k-1/2}\right).
	\end{align*}
	We deduce from \eqref{fonction caracteristique sym} the Wiener--Hopf factorisation $1-\varphi_{\mathrm{sym}}(\theta) =(1/2) \sqrt{1-e^{-i\theta}} \sqrt{1-e^{i\theta}}$, with the generating function of the strict descending ladder height $G_{\mathrm{sym}}^<(z) = 1- \sqrt{1-z}$ (the uniqueness of the Wiener--Hopf factorisation was recently proven in Theorem 3 of \cite{DSTW24}). Hence, by doing the same computations as in the beginning of the proof of Theorem \ref{th marche}, the function $h^\downarrow$ is $\nu_{\mathrm{sym}}$-harmonic on $\Z_{\ge 1}$. In other words, for all $p\ge 1 $,
	$$
	\sum_{k \in \Z} h^\downarrow(p+k) \frac{4}{4k^2-1} = 0,
	$$
	hence \eqref{eq h fleche vers le bas harmonique}. 
	
	Then, let us prove that $\nu(k-1)-\nu(-k-1)=g_k$ for all $k\ge 1$. Let $g(z)\coloneqq \sum_{k\ge 1} g_k z^k$. By \eqref{eq sommable2}, we deduce that the power series $f(z)\coloneqq \sum_{k\ge 1} f_k z^k$ has radius one and that $\theta \mapsto f(e^{i\theta}) \in \mathrm{L}^2([0,2\pi])$. By doing the computations leading to \eqref{eq f k 2} and \eqref{eq f k} in reverse order, one gets that for all $k\ge 1$,
	$$
	f_k = \frac{1}{2\pi} \int_0^{2\pi}  	\frac{1}{\sqrt{1-e^{i\theta}}\sqrt{1-e^{-i\theta}}} \left( - g(e^{i\theta}) + g(e^{-i\theta}) +e^{i\theta} - e^{-i\theta} \right) e^{-i k \theta}  d\theta.
	$$
	Moreover, for all $\theta \in \R$,
	$$
	\vert g(e^{i\theta}) - g(e^{-i\theta}) \vert =  \left\vert \sum_{k\ge 1} g_k 2i \sin(k\theta)\right\vert \le 2 \vert \theta \vert \sum_{k\ge 1} k g_k,
	$$
	so that the function $ \theta \mapsto ({1}/({\sqrt{1-e^{i\theta}}\sqrt{1-e^{-i\theta}}})) ( - g(e^{i\theta}) + g(e^{-i\theta}) +e^{i\theta} - e^{-i\theta} )$ is bounded. 
	
	As a consequence, we obtain the equality for Lebesgue almost all $\theta \in [0,2\pi]$,
	$$
	\frac{1}{\sqrt{1-e^{i\theta}}\sqrt{1-e^{-i\theta}}} \left( - g(e^{i\theta}) + g(e^{-i\theta}) +e^{i\theta} - e^{-i\theta} \right) = f(e^{i\theta}) - f(e^{-i\theta}).
	$$
	Besides, let $\varphi$ be the characteristic function of $\nu$. Since $\theta \mapsto f(e^{i\theta}) \in \mathrm{L}^2([0,2\pi])$, we can also do the computation leading to \eqref{eq nu k} in reverse order and get \eqref{eq lien phi f} for Lebesgue almost every $\theta\in [0,2\pi]$. Combining \eqref{eq lien phi f} with the above equality yields
	$$
	e^{i\theta} \varphi(\theta) - e^{-i\theta}\varphi(-\theta)  = g(e^{i\theta}) - g(e^{-i\theta})
	$$
	for Lebesgue almost every $\theta \in [0,2\pi]$. This proves that $\nu(k-1)-\nu(-k-1)=g_k$ for all $k\ge 1$. 
	The next point of the proposition comes from exactly the same ideas as in the proof of Theorem \ref{th carte o2}. The last points come from the gasket decomposition.
\end{proof}
\subsection{Examples}
Let us give two examples of models of critical $O(2)$ loop-decorated planar maps. We start with an example from Timothy Budd and then give a new example which only has faces of degree $4$ which are traversed with loops.
\subsubsection{The symmetric example of Budd} A key ingredient of the proof of Proposition \ref{proposition} is the symmetric probability distribution $\nu_{\mathrm{sym}}$ given by
$$
\forall k \in \Z, \qquad \nu_{\mathrm{sym}}(k) = {\bf 1}_{k=0} +  \frac{2}{\pi} \frac{1}{4k^2 - 1}.
$$
Here the slowly varying function of Theorem \ref{th marche} is a constant. The associated weight sequence $\bf q$ for the gasket is given by
$$
\forall k\ge 1, \qquad q_k = \left({\bf 1}_{k=1} +  \frac{2}{\pi} \frac{1}{4(k-1)^2 - 1}\right) (3\pi)^{-k+1}
$$
and $c_{\bf q} = 3\pi$, while the associated triple $(\widetilde{\bf q}, 1/c_{\bf q},2)$ is given by Remark 1 of \cite{B18}: using \eqref{eq lien q qtilde} and \eqref{eq lien q nu}, one computes for all $k\ge 1$,
$$
\tilde{q}_k = c_{\bf q}^{-k+1} (\nu(k-1)-\nu(-k-1))= (3\pi)^{1-k} \frac{2 k}{\pi(k - 3/2) (k - 1/2) (k + 1/2) (k + 3/2)} + {\bf 1}_{k=1}.
$$
From the above expression one can see that there are faces of arbitrary (even) degrees which are not traversed by loops. The partition function is given for all $\ell \ge 1$ by
$$
W^{(\ell)} = \frac{c_{\bf q}^{\ell+1}}{2} \nu(-\ell-1) = \frac{(3\pi)^{\ell+1}}{\pi} \frac{1}{4(\ell+1)^2-1}.
$$

\subsubsection{Fully packed critical $O(2)$-loop decorated quadrangulations}\label{sous-section fully packed}

\begin{figure}[h]
	\centering
	\includegraphics[width=0.43\textwidth]{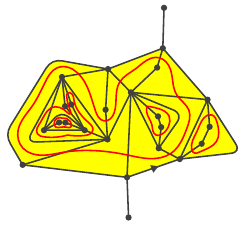}
	\hspace{0.8cm}
	\includegraphics[width=0.39\textwidth]{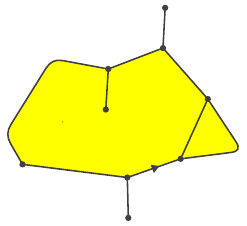}
	\caption{Left: a fully packed loop-decorated quadrangulation. Right: its gasket.}
	\label{fig fullypacked}
\end{figure}
A new simple example of weight sequence is obtained by taking $g_j=0$ for all $j\ge 1$. By \eqref{eq nu 2}, for all $k\in \Z$,
$$\nu(k) = {\bf 1}_{k=0} + \frac{1}{\pi^2} \sum_{\ell \ge 1} \left(\frac{1}{\ell-1/2} + \frac{1}{\ell+1/2}\right) \frac{4}{4(\ell-k-1)^2 -1}.$$
It clearly satisfies $\nu(k) \ge 0$ for all $k \le -1$ and one computes for all $k\ge 1$,
\begin{align*}
	&\nu(k-1)-\nu(-k-1) \\
	&= {\bf 1}_{k=1}+ \frac{1}{\pi^2}
	\sum_{\ell \ge 1} \left(\frac{1}{\ell-1/2} + \frac{1}{\ell+1/2}\right)\left( \frac{4}{4(\ell-k)^2 -1}
	-\frac{4}{4(\ell+k)^2 -1}\right)=0.
\end{align*}
In particular, $\nu(k) \ge 0$ for all $k \in \Z$. Moreover, one computes $\nu(-1) = 4/\pi^2$. Thus, by Proposition \ref{proposition}, the weight sequence ${\bf q}= (q_k)_{k\ge 1}$ defined by
$$
\forall k\ge 1, \qquad q_k  = \left({\bf 1}_{k=1} + \frac{1}{\pi^2} \sum_{\ell \ge 1} \left(\frac{1}{\ell-1/2} + \frac{1}{\ell+1/2}\right) \frac{4}{4(\ell-k)^2 -1}\right) (\pi^2/2)^{1-k}
$$
is the weight sequence of the gasket of a critical $O(2)$-decorated planar map and $c_{\bf q} = 2/\nu(-1) = \pi^2/2$, while the associated triple is $(\widetilde{\bf q}, 2/\pi^2, 2)$, where $\widetilde{\bf q}$ is the constant sequence equal to zero. Here, the $O(2)$-decorated map only has quadrangles and every quadrangle is traversed by a loop. See Figure \ref{fig fullypacked}. The partition function is given, for all $\ell\ge 1$, by
$$
W^{(\ell)} = \frac{c_{\bf q}^{\ell+1}}{2} \nu(-\ell-1) = \frac{(\pi^2/2)^{\ell+1}}{2\pi^2}  \sum_{k\ge 1} \left(\frac{1}{k-1/2} + \frac{1}{k+1/2} \right) \frac{4}{4(k+\ell)^2 -1}.
$$
In this example, the slowly varying function $L_{\bf q}(\ell)$ is equivalent to $(2/\pi^2)\log (\ell)$ as $\ell \to \infty$.

\paragraph{Acknowledgements.} Many thanks to Cyril Marzouk and Nicolas Curien for their support. I thank Elric Angot and Orphée Collin for stimulating discussions at an early stage of this work. I thank \'Elie A\"id\'ekon, William Da Silva and XingJian Hu for sharing with me their new version of \cite{AdSH24}. I also acknowledge the support of the ERC consolidator grant 101087572 ``SuPerGRandMa''. I am grateful to the anonymous referee for their careful reading.
\bibliographystyle{alpha}
\bibliography{cartesO2}

\end{document}